\documentclass[11pt]{amsart}

\usepackage{amsmath, amsthm, amscd, amsfonts, amssymb, graphicx, color}
\usepackage[bookmarksnumbered, plainpages]{hyperref}
\usepackage{graphicx}
\usepackage{epsfig}

 \newtheorem{them}{Theorem}[section]

 \newtheorem{thm}[them]{THEOREM}
 \newtheorem{cor}[them]{COROLLARY}
  
 \newtheorem{lem}[them]{LEMMA}
 
 \theoremstyle{definition}
 \newtheorem{defn}[them]{\textit{Definition}}
 \theoremstyle{remark}
 
 \numberwithin{equation}{section}
 
  \theoremstyle{Example}
 \newtheorem{ex}[them]{\textit{Example}}

 \newcommand{\A}{\mathcal{A}}
   \newcommand{\B}{\mathcal{B}}

\begin{document}

\textwidth=13cm
\setcounter{page}{1}

\renewcommand{\currentvolume}{00}

\renewcommand{\currentyear}{0000}

\renewcommand{\currentissue}{0}

\title[MODULE BIFLATNESS OF THE SECOND DUAL OF BANACH ALGEBRAS]
 {MODULE BIFLATNESS OF THE SECOND DUAL OF BANACH ALGEBRAS}

 \author[A. Bodaghi]{ABASALT BODAGHI}

\author[A. JABBARI]{ ALI JABBARI}

\thanks{}



\date{}

\dedicatory{}



\maketitle
\begin*{}
Let $\mathcal A$ be a Banach algebra. Using the concept of module biflatness, we show that the module amenability of the second dual $\mathcal A^{**}$ (with the first Arens product) necessitates the module amenability of $\mathcal A$. We give some examples of Banach algebras $\mathcal A$ such that $\mathcal A^{**}$ are module biflat, but which are not themselves module biflat.
\end*{}
\section{INTRODUCTION}
Let $\mathcal A$ be a Banach algebra and $\omega: \mathcal A \widehat
\otimes \mathcal A\longrightarrow \mathcal A$; $a\otimes b\mapsto ab$ be the canonical morphism (for emphasis, $\omega_{\mathcal A}$). Clearly, $\omega$
is an $\mathcal A$-bimodule map (i.e. a bounded linear map which preserves the module operations) with respect to the canonical bimodule structure on the
projective tensor product $\mathcal A \widehat \otimes \mathcal A$. A Banach algebra $\mathcal A$ is called {\it biprojective} if $\omega$ has a bounded right inverse which is an $\mathcal A$-bimodule map. A Banach algebra $\mathcal A$ is said to be {\it biflat} if the adjoint $\omega^*:\mathcal A^*\longrightarrow  (\mathcal A \widehat\otimes \mathcal A)^*$ has a bounded left inverse which is an $\mathcal A$-bimodule map. It is obvious that every biprojective Banach algebra is biflat. The basic properties of biprojectivity and biflatness are investigated in \cite{hel3}.

To work with amenable semigroups is in general harder than to work amenable groups \cite{day}. One major difference is that in the group case the amenability of group is related to the amenability of the Banach algebra of absolutely integrable functions on the group. The concept of amenability for Banach algebras was initiated by Johnson in \cite{jo1}.  One of the fundamental results was that $L^1(G)$ is an amenable Banach algebra if and only if $G$ is an amenable
locally compact group. Now Johnson's amenability theorem fails for semigroups. Indeed there are many commutative (and so amenable) discrete semigroups for which the corresponding semigroup algebra is not amenable (a characterization of the semigroups $S$ such that $\ell ^{1}(S)$ is amenable
is given in Theorem 10.12 of \cite{dls}). The concept of module amenability is introduced to resolve this problem. It is shown in \cite{am1} that a discrete inverse semigroup is amenable if and only if its semigroup algebra is module amenable. The notion of module amenability is closely related to module biprojectivity and biflatness.

In \cite{bod}, Bodaghi and Amini introduced a module biprojective and module biflat Banach algebra which is a Banach module over another Banach algebra
with compatible actions. For every inverse semigroup $S$ with subsemigroup $E$ of idempotents, they showed that $\ell ^{1}(S)$ is
module biprojective, as an $\ell ^{1}(E)$-module, if and only if
an appropriate group homomorphic image $G_S$ of $S$ is
finite. They also proved that module biflatness of $\ell ^{1}(S)$ is equivalent to the amenability of the underlying semigroup $S$. Some examples  of Banach algebras which are module biprojective (biflat), but not biprojective (biflat) are given in \cite{bod}.

It is shown in \cite[Proposition 3.6]{am2} that if $\mathcal A$
is a commutative Banach ${\mathfrak A}$-bimodule such that $\mathcal
A^{**}$ is ${\mathfrak A}$-bimodule amenable, then $\mathcal A$ is
module amenable. In part three of this paper we improve this
result by assuming a weaker condition on $\mathcal A$, using our
results on module biflatness. Among many other things, we bring some examples of Banach algebras such that their second dual are module biflat, but which are not themselves module biflat.



\section{MODULE BIFLATNESS}

Let ${\mathcal A}$ and ${\mathfrak A}$ be
Banach algebras. Suppose that ${\mathcal A}$ is a Banach ${\mathfrak
A}$-bimodule with compatible actions, that is
$$\alpha\cdot(ab)=(\alpha\cdot a)b,
\,\,(ab)\cdot\alpha=a(b\cdot\alpha) \hspace{0.3cm}(a,b \in
{\mathcal A},\alpha\in {\mathfrak A}).
$$
Let $X$ be a Banach ${\mathcal A}$-bimodule and a
 Banach ${\mathfrak A}$-bimodule with compatible actions, that is
$$\alpha\cdot(a\cdot x)=(\alpha\cdot a)\cdot x,
\hspace{0.5cm}a\cdot(\alpha\cdot x)=(a\cdot\alpha)\cdot x, $$
$$(\alpha\cdot x)\cdot a=\alpha\cdot(x\cdot a) \hspace{0.5cm}(a \in{\mathcal
A},\alpha\in {\mathfrak A},x\in{X} )$$
and similarly for the right or two-sided actions. Then we say that
$X$ is a Banach ${\mathcal A}$-${\mathfrak A}$-module. A Banach ${\mathcal A}$-${\mathfrak A}$-module $X$ is called {\it commutative} if
$\alpha\cdot x=x\cdot\alpha$ for all $\alpha\in {\mathfrak
A}$ and $x\in X $. If $X
$ is a (commutative) Banach ${\mathcal A}$-${\mathfrak
A}$-module, then so is $X^*$.

Let $X, Y$ be Banach $\mathcal A$-$\mathfrak A$-modules. We say that a map $\phi:
X\longrightarrow Y$ is a left $\mathcal A$-$\mathfrak A$-module homomorphism if it is an
$\mathfrak A$-bimodule homomorphism and a left $\mathcal A$-module homomorphism,
that is
$$\phi(\alpha\cdot x)=\alpha\cdot\phi(x),\quad
\phi(x\cdot\alpha)=\phi(x)\cdot\alpha,\quad \phi(a\cdot x)=a\cdot
\phi(x),$$
for all $\alpha\in \mathfrak A, a\in \mathcal A$ and  $x\in X$. Right $\mathcal A$-$\mathfrak A$-module homomorphisms and (two-sided) $\mathcal A$-$\mathfrak A$-module homomorphisms are defined in a similar fashion.

Let $\mathcal A$ be a commutative $\mathfrak A$-bimodule and
act on itself by multiplication from both sides, that is
$$a\cdot b=ab,\quad b\cdot a=ba \qquad(a,b\in \mathcal A).$$
Then it is also
a Banach ${\mathcal A}$-${\mathfrak A}$-module. If ${\mathcal A}$ is a Banach $\mathfrak A$-bimodule with
compatible actions, then so are the dual space $\mathcal A^*$ and
the second dual space $\mathcal A^{**}$. If moreover $\mathcal A$
is a commutative $\mathfrak A$-bimodule, then $\mathcal A^*$ and the
$\mathcal A^{**}$ are commutative ${\mathcal A}$-${\mathfrak
A}$-modules. The canonical images of $a\in\mathcal{A}$ and $\mathcal{A}$ in $\mathcal{A}^{**}$ will be denoted by $\hat{a}$ and $\hat{\mathcal{A}}$, respectively.

Let ${\mathcal A}$ be a Banach ${\mathfrak
A}$-bimodule with compatible actions, and let ${\mathcal A}\widehat{\otimes} _{\mathfrak A} {\mathcal A}$ be the module projective tensor product of $\mathcal{A}$ and $\mathcal{A}$. Then ${\mathcal A}\widehat{\otimes} _{\mathfrak A} {\mathcal A}$ is isomorphic
to the quotient space $(\mathcal A \widehat{\otimes} \mathcal
A)/{I_{\mathcal A}}$, where $I_{\mathcal A}$ is the closed linear span of $\{ a\cdot\alpha \otimes b-a
\otimes\alpha \cdot b : \alpha\in {\mathfrak A},a,b\in{\mathcal A}\}$ \cite{ra}. Also consider the closed ideal
$J_{\mathcal A}$ of ${\mathcal A}$ generated by elements of the form $
(a\cdot\alpha) b-a(\alpha \cdot b)$ for $ \alpha\in {\mathfrak
A},a,b\in{\mathcal A}$. We shall denote $I_{\mathcal A}$ and
$J_{\mathcal A}$ by $I$ and $J$, respectively, if there is no risk of confusion.
Then $I$ is an ${\mathcal A}$-submodule and an ${\mathfrak A}$-submodule of
$\mathcal A \widehat{\otimes} \mathcal A$, $J$ is an ${\mathcal A}$-submodule and an ${\mathfrak A}$-submodule of $\mathcal A$, and both of the quotients $\mathcal A \widehat{\otimes}_{\mathfrak A} \mathcal A$ and $\mathcal A/J$ are
${\mathcal A}$-bimodules and ${\mathfrak A}$-bimodules. Also,
$\mathcal A/J$ is a Banach $\mathcal A$-${\mathfrak A}$-module
 when ${\mathcal A}$ acts on ${\mathcal
A}/J$ canonically. Let $\tilde{\omega} : \mathcal{A}\widehat{\otimes}_{\mathfrak{A}}\mathcal{A} =
{(\mathcal{A}\widehat{\otimes}\mathcal{A}})/{{I}}\longrightarrow
{\mathcal{A}}/{{J}}$ be induced product map by $\omega$, i.e., $\tilde{\omega}(a\otimes
b+{I})=ab+{J}$.

Let $I$ and $J$ be the closed ideals as in the above. Then $(\mathcal A \widehat \otimes \mathcal A)/{I}$
is not a $\mathcal A/J$-bimodule, in general. It is shown in \cite{bod} that if $\mathcal A/J$ has a right bounded approximate identity,
then $(\mathcal A \widehat \otimes \mathcal A)/{I}$ is an $\mathcal A/J$-bimodule when ${\mathfrak A}$ acts on $\mathcal A$ trivially
from the left or the right (see also \cite[Lemma 3.13]{bab}). Recall that a Banach algebra ${\mathfrak A}$ acts trivially on
$\mathcal A$ from the left if there is a continuous linear
functional $f$ on ${\mathfrak A}$ such that  $\alpha\cdot
a=f(\alpha)a$ for all $\alpha\in
\mathfrak A$ and $a\in \mathcal A$. The trivial right action is defined, similarly (see \cite {amb}). From now on, when we
consider $(\mathcal A \widehat \otimes \mathcal A)/{I}$ as an
$\mathcal A/J$-bimodule, we have assumed that the above conditions are satisfied (for more details see \cite{bod}).

\begin{defn}\label{def1} \cite{bod} Let $\mathfrak A$ and $\mathcal A$ be Banach algebras. Then $\mathcal A$ is called {\it module
biprojective} \text{(}as an $\mathfrak A$-bimodule\text{)} if $\widetilde{\omega}$  has a bounded right inverse which is an
 ${\mathcal A}/J$-${\mathfrak A}$-module homomorphism.
\end{defn}
\begin{defn}\label{def2} \cite{bod} Let $\mathfrak A$ and $\mathcal A$ be Banach algebras. Then  $\mathcal A$ is called {\it module biflat}
\text{(}as an $\mathfrak A$-bimodule\text{)} if $ \widetilde{\omega}^* $ has a
bounded left inverse which is an ${\mathcal A}/J$-${\mathfrak
A}$-module homomorphism.
\end{defn}

Let $\mathcal A$ and $\B$ be Banach algebras. The weak$^*$ operator topology ($W^*OT$) on $\mathcal{L}(\A,\B^*)$ is the locally convex topology determined by the seminorms $\{p_{a,b}:a\in\A, b\in\B\}$ where $p_{a,b}(f)=|\langle f(a),b\rangle|$ in which $f\in \mathcal{L}(\A,\B^*)$.

\begin{thm}\label{2.1}
Let $\A$ be a Banach $\mathfrak{A}$-bimodule. Then the following statements are equivalent:
\begin{itemize}
  \item[(i)] $\A$ is module biflat;
  \item[(ii)] There is an $\A/J$-$\mathfrak{A}$-module homomorphism $\rho:\A/J\longrightarrow((\A\widehat{\otimes}\A)/I)^{**}$ such that $\widetilde{\omega}^{**}\circ\rho$ is the canonical embedding from $\A/J$ into $\A^{**}/J^{\bot\bot}$.
       \item[(iii)] There is a net $(\widetilde{\omega}_\gamma)$ of uniformly bounded $\A/J$-$\mathfrak{A}$-module homomorphisms from $(\A/J)^*$ into $((\A\widehat{\otimes}\A)/I)^{*}$ such that $\lim_\gamma W^*OT-\widetilde{\omega}_\gamma\circ\widetilde{\omega}^{*}=id_{(\A/J)^*}$.
\end{itemize}
\end{thm}
\begin{proof}
(i)$\Rightarrow$(ii). Let $\A$ be module biflat. Then there exists a bounded left inverse $\pi:((\A\widehat{\otimes}\A)/I)^{*}\longrightarrow(\A/J)^*$ which is also an $\A/J$-$\mathfrak{A}$-module homomorphism such that $\pi\circ\widetilde{\omega}^{*}$ is the identity mapping. Put $\rho:=\pi^*|_{\A/J}$. Thus the mapping $\rho:\A/J\longrightarrow((\A\widehat{\otimes}\A)/I)^{**}$ is an $\A/J$-$\mathfrak{A}$-module homomorphism. Therefore $\widetilde{\omega}^{**}\circ\rho$ is the canonical embedding from $\A/J$ into $\A^{**}/J^{\bot\bot}$.

(ii)$\Rightarrow$(i). We firstly note that $\rho^*:((\A\widehat{\otimes}\A)/I)^{***}\longrightarrow(\A/J)^*$ is an $\A/J$-$\mathfrak{A}$-module homomorphism. Suppose that $\widetilde{\rho}$ is the restriction of $\rho^*$ on $((\A\widehat{\otimes}\A)/I)^{*}$ which is also an $\A/J$-$\mathfrak{A}$-module homomorphism. For every $f\in(\A/J)^*$ and $a\in\A$ we have
\begin{eqnarray*}
\langle \widetilde{\rho}\circ\widetilde{\omega}^*(f),a+J\rangle=\langle \rho(a+J), \widetilde{\omega}^*(f)\rangle=\langle \widetilde{\omega}^{**}\circ\rho(a+J), f\rangle=\langle f,a+J\rangle.
  \end{eqnarray*}
This means that $\widetilde{\omega}^*$ has a bounded left inverse which is $\A/J$-$\mathfrak{A}$-module homomorphism.

(i)$\Rightarrow$(iii) This is trivial.

(iii)$\Rightarrow$(i) Since every bounded subsets of $\mathcal{L}(((\A\widehat{\otimes}\A)/I)^{*},(\A/J)^*)$ are $W^*OT$ compact, the net $(\widetilde{\omega}_\gamma)$  has a $W^*OT$ point limit, say $\pi$. It is easy to check that $\pi$ is an $\A/J$-$\mathfrak{A}$-module homomorphism and $\pi\circ\widetilde{\omega}^{*}$ is the identity mapping.
\end{proof}

Let $X$, $Y$ and $Z$ be Banach ${\mathcal A}/J$-${\mathfrak
A}$-modules. Then the short exact sequence
\begin{equation*}\{0\}\longrightarrow X
\stackrel{\varphi}{\longrightarrow}Y
\stackrel{\psi}{\longrightarrow} Z\longrightarrow
\{0\}\end{equation*} is {\it admissible} if $\psi$ has a bounded
right inverse which is ${\mathfrak A}$-module homomorphism, and
{\it splits} if $\psi$ has a bounded right inverse which is a
${\mathcal A}/J$-${\mathfrak A}$-module homomorphism.

\begin{thm}\label{2.2}
Let $\A$ be a Banach $\mathfrak{A}$-bimodule with trivial left action, and let $\B$ two-sided closed ideal of $\A$. If  $\B/J_\B$ has a bounded approximate identity which is also a bounded approximate identity for $\A/J$, then $\A$ is module biflat.
\end{thm}
\begin{proof}We firstly note that when $\B/J_\B$ has a right bounded approximate identity,
then $(\mathcal{B}\widehat{\otimes}\B)/I_\B$  is $\B/J_\B$-bimodule if ${\mathfrak A}$ acts on $\mathcal B$ trivially
from left \cite{bod}. Consider the short exact sequence of $\B/J_\B$-$\mathfrak{A}$-module homomorphisms as follows:
\begin{equation}\label{s.2.1}
         0\longrightarrow \ker(\widetilde{\omega}_\B)\stackrel{\imath}{\longrightarrow}(\mathcal{B}\widehat{\otimes}\B)/I_\B\stackrel{\widetilde{\omega}_\B}\longrightarrow \mathcal{B}/J_\B\longrightarrow0.
\end{equation}
Then the following first and second duals of (\ref{s.2.1}) are $\B/J_\B$-$\mathfrak{A}$-module homomorphisms
\begin{equation}\label{s.2.2}
         0\longrightarrow (\mathcal{B}/J_\B)^* \stackrel{\widetilde{\omega}_\B^*}{\longrightarrow}((\mathcal{B}\widehat{\otimes}\B)/I_\B)^*\stackrel{\imath^*}\longrightarrow \ker(\widetilde{\omega}_\B)^*\longrightarrow0
\end{equation}
\begin{equation}\label{s.2.3}
        0\longrightarrow \ker(\widetilde{\omega}_\B)^{**}\stackrel{\imath^{**}}{\longrightarrow}((\mathcal{B}\widehat{\otimes}\B)/I_\B)^{**}\stackrel{\widetilde{\omega}_\B^{**}}\longrightarrow (\mathcal{B}/J_\B)^{**}\longrightarrow0.
\end{equation}
By Lemma 2.2 of \cite{bod}, the short exact sequences  (\ref{s.2.2})  and (\ref{s.2.3}) are admissible. Therefore $\widetilde{\omega}_\B^{**}$ has a bounded right inverse, say $\rho$, which is a $\B/J_\B$-$\mathfrak{A}$-module homomorphism. Let $i:\B\longrightarrow\A$ be the canonical embedding. Thus $i$ induces the mapping $i\otimes_\mathfrak{A}i:(\B\widehat{\otimes}\B)/I_\B\longrightarrow\A\widehat{\otimes}\A/I$ as the usual definition. Put $\widetilde{\rho}:=(i\otimes_\mathfrak{A}i)^{**}\circ\rho$. Assume that $e_j+J_\B$ is the bounded approximate identity of $\B/J_\B$. Define $\overline{\rho}:\A/J\longrightarrow((\A\widehat{\otimes}\A)/I)^{**}$ by $\overline{\rho}(a+J)=\text{weak}^*-\lim_\alpha\widetilde{\rho}(e_j+J_\B)\cdot(a+J)$. By definition of $\overline{\rho}$ and Cohen's factorization theorem, $\overline{\rho}$ is an $\A/J$-$\mathfrak{A}$-module homomorphism. Thus we have
\begin{align}
 \nonumber
  \widetilde{\omega}^{**}(\overline{\rho}(a+J)) &= \widetilde{\omega}^{**}(\text{weak}^*-\lim_\alpha\widetilde{\rho}(e_j +J_\B)\cdot(a+J)) \\
    \nonumber
    &=\text{weak}^*-\lim_\alpha\widetilde{\omega}^{**}(\widetilde{\rho}(e_j+J_\B)\cdot(a+J))\\
    \nonumber
  &= \text{weak}^*-\lim_\alpha\widetilde{\omega}^{**}_\B(\widetilde{\rho}(e_j+J_\B))\cdot(a+J)\\
  \nonumber
  &=\lim_\alpha (e_j+J_\B)\cdot(a+J)=a+J.
\end{align}
Applying Theorem \ref{2.1}, we observe that $\A$ is module biflat.
\end{proof}
For a discrete semigroup $S$, $\ell ^{\infty}(S)$ is the Banach
algebra of bounded complex-valued functions on $S$ with the
supremum norm and pointwise multiplication. For each $a \in S$
and $f \in \ell ^{\infty}(S)$, let $l_af$ and $r_af$ denote the
left and the right translations  of $f$ by $a$, that is
$(l_af)(s)=f(as)$ and $(r_af)(s)=f(sa)$, for each $s \in S$. Then
a linear functional $m\in (\ell ^{\infty}(S))^*$ is called a {\it
mean} if $\| m\|=\langle m,1\rangle=1$;  $m$ is called a {\it
left} ({\it right}) {\it invariant mean} if $m(l_af)=m(f)
\,(m(r_af)=m(f)$, respectively) for all $s \in S$ and  $f \in
\ell ^{\infty}(S)$. A discrete semigroup $S$ is called {\it
amenable} if there exists a mean $m$ on $\ell ^{\infty}(S)$ which
is both left and right invariant (see \cite{dun}). An {\it
inverse semigroup} is a discrete semigroup $S$ such that for each
$s\in S$, there is a unique element $s^*\in S$ with $ss^*s=s$ and
$s^*ss^*=s^*$. Elements of the form $ss^*$ are called {\it
idempotents} of $S$. For an inverse semigroup $S$, a left
invariant mean on $\ell ^{\infty}(S)$ is right invariant and vice
versa.

Let $S$ be an inverse
semigroup with the set of idempotents $E$ (or $E_S$), where the order of $E$
is defined by
$$e\leq d \Longleftrightarrow ed=e \hspace{0.3cm}(e,d \in
E).$$
Since $E$ is a commutative subsemigroup of $S$
\cite[Theorem V.1.2]{ho}, actually a semilattice, $\ell
^{1}(E)$ could be regarded as a commutative subalgebra of $ \ell
^{1}(S)$, and thereby $ \ell ^{1}(S)$ is a Banach algebra and a
Banach $ \ell ^{1}(E)$-bimodule with compatible actions \cite{am1}.
Here, for technical reason, we assume that $ \ell ^{1}(E)$ acts on $ \ell ^{1}(S)$ by
multiplication from right and trivially from left, that is
\begin{eqnarray}\label{e1}\delta_e\cdot\delta_s = \delta_s, \,\,\delta_s\cdot\delta_e = \delta_{se} =
\delta_s * \delta_e \hspace{0.3cm}(s \in S,  e \in E).
\end{eqnarray}
In this case, the ideal $J$  is the closed linear
span of $\{\delta_{set}-\delta_{st} :\, s,t \in S,  e \in E\}.$ We consider an equivalence relation on $S$ as follows:
$$s\approx t \Longleftrightarrow \delta_s-\delta_t \in J \hspace{0.2cm} (s,t \in
S).$$
For an inverse semigroup $S$, the quotient ${S}/{\approx}$ is a
discrete group (see \cite{am2} and \cite{pou}). Indeed,
${S}/{\approx}$ is isomorphic to the maximal group  homomorphic
image $G_S$ \cite{mn} of $S$ \cite{pou2}. In particular, $S$ is
amenable if and only if $S/\approx$ is amenable \cite{dun, mn}. As
in \cite[Theorem 3.3]{ra1}, we see that $\ell^{1}(S)/J\cong {\ell ^{1}}(G_S)$. Note that ${\ell ^{1}}(G_S)$ has an identity.

Let $H$ be a subsemigroup of $S$, $E'$ be the set of idempotents in $H$ and $J'$  be the closed linear
span of $\{\delta_{set}-\delta_{st} \quad s,t \in H,  e \in E'\}.$ Clearly, $E'\subseteq E$ and $J'\subseteq J$. If $\approx$ is a equivalence relation on $S$ as defined above, then it is also a equivalence relation on $H$. By the above statements we have the following result.
\begin{cor}
Let $S$ be an inverse semigroup, and let $H$ be an ideal in $S$. Assume that the identity of $\ell^1(G_H)$  is an identity for $\ell^1(G_S)$. Then $S$ is amenable.
\end{cor}
\begin{proof}
By Theorem \ref{2.2}, $\ell^1(S)$ is module biflat. Now, it follows from \cite[Theorem 3.2]{bod} that $S$ is amenable .
\end{proof}

Let $(\A_\gamma)_{\gamma\in\Gamma}$ be a family of Banach $\mathfrak{A}$-bimodules. For every $\gamma\in\Gamma$, we consider the $\mathfrak{A}$-module mappings $\pi_{\gamma}:((\A_{\gamma}\otimes\A_{\gamma})/I_{\gamma})^*\longrightarrow(\A_{\gamma}/J_{\gamma})^*$ and $\omega_{\gamma}:((\A_{\gamma}\otimes\A_{\gamma})/I_{\gamma})\longrightarrow(\A_{\gamma}/J_{\gamma})$.
\begin{thm}
Let $\A$ be a Banach $\mathfrak{A}$-bimodule, let $(\A_\gamma)_{\gamma\in\Gamma}$ be a family of module biflat closed ideals in $\A$, and let $\A$ has a bounded approximate identity contained in $\cup_\gamma \A_\gamma$. If  $\sup_\gamma \|\pi_\gamma\|<\infty$ such that $\pi_\gamma\circ\widetilde{\omega}_\gamma^*$ is identity mapping, for every $\gamma$, then $\A$ is module biflat.
\end{thm}
\begin{proof}
Let $(e_\lambda)_\lambda$ be a bounded approximate identity for $\A$ which contained in $\cup_\gamma \A_\gamma$, let $\eta=(F,\Phi,\varepsilon)$ where $F\subset\A$ and $\Phi\in\A^*$ are finite sets. Choose $\gamma_0\in\Gamma$ such that $e_{\lambda_0}\in\A_{\gamma_0}$ and $\|ae_{\gamma_0}+J_{\gamma_0}-(a+J)\|<\varepsilon$ for every $\varepsilon>0$. Consider the mappings $i:\A_{\gamma_0}\longrightarrow\A$ and $\rho_0:\A/J\longrightarrow\A_{\gamma_0}/J_{\gamma_0}$ by $i_0(a)=a$ and $\rho_0(a+J)=ae_{\lambda_0}+J_{\gamma_0}$ for all $a\in\A$. Since every $\A_\gamma$ is module biflat, there is a bounded left inverse $\A_{\gamma_0}/J_{\gamma_0}$-$\mathfrak{A}$-module homomorphism $\pi_{\gamma_0}:((\A_{\gamma_0}\otimes\A_{\gamma_0})/I_{\gamma_0})^*\longrightarrow(\A_{\gamma_0}/J_{\gamma_0})^*$ such that $\pi_{\gamma_0}\circ\widetilde{\omega}_{\A_{\gamma_0}}^*=id_{(\A_{\gamma_0}/J_{\gamma_0})^*}$. It is easily verified that $(i\otimes_\mathfrak{A}i)^*\circ\widetilde{\omega}^*|_{\A_{\gamma_0}}=\widetilde{\omega}_{\A_{\gamma_0}}^*$. Letting $\pi_\eta:=\rho_0^*\circ\pi_{\gamma_0}\circ(i\otimes_\mathfrak{A}i)^*$, we get
\begin{align*}
| \langle a+J, \pi_\eta\circ\widetilde{\omega}^*(f)-f\rangle | &= |\langle a+J, \rho_0^*\circ\pi_{\gamma_0}\circ(i\otimes_\mathfrak{A}i)^*\circ\widetilde{\omega}^*(f)-f\rangle| \\
   &\leq|\langle \rho_0(a+J),\pi_{\gamma_0}\circ(i\otimes_\mathfrak{A}i)^*\circ\widetilde{\omega}^*(f)-f\rangle|\\
   &+ |\langle \rho_0(a+J)-(a+J),f\rangle|\\
   &= |\langle ae_{\lambda_0}+J_{\gamma_0},\pi_{\gamma_0}\circ \widetilde{\omega}_{\A_{\gamma_0}}^*(f|_{\A_{\gamma_0}/J_{\gamma_0}})-f|_{\A_{\gamma_0}/J_{\gamma_0}}\rangle|\\
   &+|\langle ae_{\lambda_0}+J_{\gamma_0}-(a+J),f\rangle|\\
   &\rightarrow 0.
\end{align*}
Thus $\lim W^*OT\pi_\eta\circ\widetilde{\omega}^*=id_{(\A/J)^*}$. According to our assumption $\pi_\eta$ is uniformly bounded. The proof now is complete by Theorem \ref{2.1}.
\end{proof}

Let $\A$ and $\B$ be Banach $\mathfrak{A}$-bimodules. Consider the projective module tensor product $\A\widehat{\otimes}_\mathfrak{A}\B$. This product is the quotient of the usual projective tensor product $\A\widehat{\otimes} \B$ by the closed ideal $\mathcal{I}$ generated by $a\cdot\alpha\otimes b-a\otimes\alpha \triangleleft b$ for every $a\in\A$, $b\in\B$ and $\alpha\in\mathfrak{A}$. We denote the space of all bounded $\mathfrak{A}$-module maps from $\A$ into $\B$ by $\mathcal{L}_\mathfrak{A}(\A, \B)$. By the following module actions $\mathcal{L}_\mathfrak{A}(\A, \B)$ is a Banach $\A\widehat{\otimes}_\mathfrak{A} \B$-$\mathfrak{A}$-bimodule
\begin{equation}\label{ma1}
    \langle (a\otimes b)\star T,c\rangle=bT(ca),\hspace{0.5cm}\hspace{0.5cm}\langle \alpha\bullet T,a\rangle=\langle T,a\cdot\alpha\rangle,
\end{equation}
\begin{equation}\label{ma2}
    \langle T\star(a\otimes b),c\rangle= T(ac)b\hspace{0.5cm}\emph{\emph{and}}\hspace{0.5cm}\langle T\bullet\alpha,a\rangle=\langle T,\alpha \cdot a\rangle,
\end{equation}
for every $a, c\in\A$, $b\in\B$ and $\alpha\in\mathfrak{A}$. We also denote the Banach $\A\widehat{\otimes}_\mathfrak{A} \B$-$\mathfrak{A}$-bimodule, $\mathcal{L}_\mathfrak{A}(\A, \B)$ by the above actions by $\widetilde{\mathcal{L}}_\mathfrak{A}(\A, \B)$. Similarly,  $\mathcal{L}_\mathfrak{A}(\B, \A)$ is a Banach $\A\widehat{\otimes}_\mathfrak{A} \B$-$\mathfrak{A}$-bimodule which we denote it by $\overline{\mathcal{L}}_\mathfrak{A}(\B, \A)$ with the actions given by
\begin{equation}\label{ma3}
    \langle (a\otimes b)\star T,c\rangle= aT(cb),\hspace{0.5cm}\hspace{0.5cm}\langle \alpha\bullet T,b\rangle=\langle T,b\triangleleft\alpha\rangle,
\end{equation}
\begin{equation}\label{ma4}
    \langle  T\star(a\otimes b),c\rangle= T(ac)b \hspace{0.5cm}\emph{\emph{and}}\hspace{0.5cm}\langle T\bullet\alpha,b\rangle=\langle T,\alpha \triangleleft b\rangle,
\end{equation}
for every $a\in\A$, $b,c\in\B$ and $\alpha\in\mathfrak{A}$. We have $\mathcal{L}_\mathfrak{A}(\A, \B^*)\cong(\A\widehat{\otimes}_\mathfrak{A}\B)^*$, and similarly $\mathcal{L}_\mathfrak{A}(\B, \A^*)\cong(\B\widehat{\otimes}_\mathfrak{A}\A)^*$.  Let $\lambda\in(\A\widehat{\otimes}\B)^*$. Define $\widetilde{T}_\lambda\in\mathcal{L}_\mathfrak{A}(\A, \B^*)$ and $\overline{T}_\lambda\in\mathcal{L}_\mathfrak{A}(\B, \A^*)$ as follow:
\begin{equation}\label{ii}
    \langle \widetilde{T}_\lambda(a),b\rangle=\langle \lambda,a\otimes b\rangle\hspace{0.5cm}\emph{\emph{and}}\hspace{0.5cm} \langle \overline{T}(b),a\rangle=\langle\lambda, a\otimes b\rangle.
\end{equation}
Consider the following isometric $\A\widehat{\otimes}_\mathfrak{A} \B$-$\mathfrak{A}$-bimodule isomorphisms maps
\begin{equation}\label{is1}
    (\A\widehat{\otimes}_\mathfrak{A}\B)^*\longrightarrow\widetilde{\mathcal{L}}_\mathfrak{A}(\A, \B^*),\hspace{0.5cm} \lambda\mapsto\widetilde{T}_\lambda
\end{equation}
and
\begin{equation}\label{is2}
    (\A\widehat{\otimes}_\mathfrak{A}\B)^*\longrightarrow\overline{\mathcal{L}}_\mathfrak{A}(\B, \A^*),\hspace{0.5cm} \lambda\mapsto\overline{T}_\lambda.
\end{equation}
By the relations (\ref{ma1})-(\ref{is2}), we can prove the following theorem.
\begin{thm}
Let $\A$ be an $\mathfrak{A}$-bimodule. If $\mathcal A$ is module biflat, then $(\widetilde{\omega}{\otimes}_\mathfrak{A}\widetilde{\omega})^*$ has a bounded left inverse.
\end{thm}
\begin{proof}
Assume that $\pi:(\A\widehat{\otimes}_\mathfrak{A}\A)^*\longrightarrow(\A/J)^*$ is a left inverse for $\widetilde{\omega}^*:(\A/J)^*\longrightarrow(\A\widehat{\otimes}_\mathfrak{A}\A)^*$ such that $\pi\circ\widetilde{\omega}^*=$id$_{(\A/J)^*}$. Consider
\begin{equation}\label{}
   \widetilde{\omega}{\otimes}_\mathfrak{A}\widetilde{\omega}:(\A\widehat{\otimes}_\mathfrak{A}\A)
    {\widehat{\otimes}}_\mathfrak{A}(\A\widehat{\otimes}_\mathfrak{A}\A)\longrightarrow(\A/J)\widehat{\otimes}_\mathfrak{A}(\A/J)
\end{equation}
 as the usual definition. We wish to show that $(\widetilde{\omega}{\otimes}_\mathfrak{A}\widetilde{\omega})^*$ has a bounded left inverse $\Lambda$
 such that $\Lambda\circ(\widetilde{\omega}{\otimes}_\mathfrak{A}\widetilde{\omega})^*=id_{((\A/J)\widehat{\otimes}_\mathfrak{A}(\A/J))^*}$. Put $\Lambda=\pi{\otimes}_\mathfrak{A}\pi$. It is clear that $\Lambda$ is bounded. We have
  \begin{align}\label{is3}
  \nonumber
 &  \Lambda\circ(\widetilde{\omega}{\otimes}_\mathfrak{A}\widetilde{\omega})^*\\ \nonumber &:((\A/J)\widehat{\otimes}_\mathfrak{A}(\A/J))^*\cong\overline{\mathcal{L}}_\mathfrak{A}(\A/J,(\A/J)^*)\stackrel{\overline{T}_\lambda\mapsto\omega^*\circ \overline{T}_\lambda}{\longrightarrow}\overline{\mathcal{L}}_\mathfrak{A}((\A/J),(\A\widehat{\otimes}_\mathfrak{A}\A)^*) \\
    \nonumber
&\cong \widetilde{\mathcal{L}}_\mathfrak{A}(\A\widehat{\otimes}_\mathfrak{A}\A,(\A/J)^*) \stackrel{\widetilde{T}_\lambda\mapsto\omega^*\circ \widetilde{T}_\lambda}{\longrightarrow}\widetilde{\mathcal{L}}_\mathfrak{A}(\A\widehat{\otimes}_\mathfrak{A}\A,(\A\widehat{\otimes}_\mathfrak{A}\A)^* ) [\emph{\emph{ by}} ~(\ref{is1})~ \emph{\emph{and}}~ (\ref{is2})] \\
     \nonumber
     &  \stackrel{\widetilde{T}_\lambda\mapsto\pi\circ \widetilde{T}_\lambda}{\longrightarrow}\widetilde{\mathcal{L}}_\mathfrak{A}\Big{(}\A\widehat{\otimes}_\mathfrak{A}\A,(\A/J)^*\Big{)}\cong\overline{\mathcal{L}}_\mathfrak{A}(\A/J,(\A\widehat{\otimes}_\mathfrak{A}\A)^*)\\
      &   \stackrel{\overline{T}_\lambda\mapsto\pi\circ \overline{T}_\lambda}{\longrightarrow}\overline{\mathcal{L}}_\mathfrak{A}(\A/J,(\A/J)^*)\cong((\A/J)\widehat{\otimes}_\mathfrak{A}(\A/J))^*.
  \end{align}
By using (\ref{is3}), we conclude that $\pi{\otimes}_\mathfrak{A}\pi$ is a left inverse of $(\widetilde{\omega}{\otimes}_\mathfrak{A}\widetilde{\omega})^*$.
\end{proof}



\section{MODULE AMENABILITY OF THE SECOND DUAL}
In this section we explore conditions under which $\mathcal
A^{**}$ is module amenable.
Recall that a Banach algebra ${\mathcal A}$
is module amenable (as an ${\mathfrak A}$-bimodule) if $\mathcal H^1_{\mathfrak A}(\mathcal A,
X^{*})=\{0\}$, for each commutative Banach ${\mathcal
A}$-${\mathfrak A}$-module $X$, where $\mathcal H^1_{\mathfrak A}(\mathcal A,
X^{*})$ is the first ${\mathfrak A}$-module cohomology group of
${\mathcal A}$ with coefficients in $X^*$ \cite{am1}.

We denote by $\square$ the first Arens product on $\mathcal{A}^{**}$, the second dual of $\mathcal{A}$. Here and subsequently, we assume that $\mathcal A^{**}$ is equipped with the first Arens product.

Let $\widetilde{\omega}$ be as in section 2. Let $I$ and $J$ be the corresponding closed ideals
of $\mathcal A \widehat \otimes \mathcal A$ and $\mathcal A$,
respectively. Consider the homomorphism
$$ \widetilde{\omega}^*: ({\mathcal A}/{J})^*=J^{\perp} \longrightarrow
 ({\mathcal A}\widehat
\otimes _{\mathfrak A} {\mathcal A})^*={\mathcal L}_{\mathfrak
 A}({\mathcal A},{\mathcal A^*}),$$
$$\langle[\widetilde{\omega}^*(f)](a),b\rangle=f(ab+J), \hspace{0.3cm}f \in ({\mathcal
A}/{J})^*.$$

Let ${\mathcal A}^{**}\widehat \otimes _{\mathfrak A} {\mathcal
A}^{**}$ be the projective module tensor product of ${\mathcal
A}^{**}$ and ${\mathcal A}^{**}$, that is ${\mathcal
A}^{**}\widehat \otimes _{\mathfrak A} {\mathcal
A}^{**}=({\mathcal A^{**} \widehat \otimes \mathcal
A^{**}})/{M}$, where $M$ is a closed ideal generated by elements
of the form $F \cdot\alpha\otimes G-F \otimes \alpha\cdot G$ for
$ \alpha\in {\mathfrak A},F,G\in{\mathcal A}^{**}$. Consider
 $$\widehat{\omega}: {\mathcal A}^{**}\widehat
\otimes _{\mathfrak A} {\mathcal A}^{**} \longrightarrow
{\mathcal A^{**}}/{N}; \quad (F \otimes G)+M\mapsto F \square G+
N,$$ where $ N $ is the closed ideal of ${\mathcal A}^{**}$
generated by $\widehat{\omega}(M)$. We know that
$$ \widehat{\omega}^*:({\mathcal A^{**}}/N)^* \longrightarrow
 ({\mathcal A}^{**}\widehat
\otimes _{\mathfrak A} {\mathcal A}^{**})^*={\mathcal
L}_{\mathfrak A}(\mathcal A^{**},\mathcal A^{***})$$
$$\langle[\widehat{\omega}^*(\phi)]F,G\rangle=\phi(F\square G+ N), \hspace{1cm} (F,G\in{\mathcal A}^{**}), $$
where $\phi\in {\mathcal A^{**}}/N$. Let $T \in ({\mathcal
A}\widehat \otimes _{\mathfrak A} {\mathcal A})^*={\mathcal
L}_{\mathfrak A}({\mathcal A},{\mathcal A^*})$ and $T^*$ ,$T^{**}$

be the first and second conjugates of $T$, then $T^{**} \in
({\mathcal A}^{**}\widehat \otimes _{\mathfrak A} {\mathcal
A}^{**})^{*}={\mathcal L}_{\mathfrak A}({\mathcal
A^{**}},{\mathcal A}^{***})$ with $\langle
T^{**}(F),G\rangle=\lim_j\lim_k\langle T(a_j),b_k)\rangle$, where
$(a_j),(b_k)$ are bounded nets in $ \mathcal A$
 such that  $ \hat a_j \stackrel{w^*}{\longrightarrow}F$
and $ \hat b_k \stackrel{w^*}{\longrightarrow}G$.

\begin{lem}
There exists an $\mathcal A$-$\mathfrak A$-module homomorphism
$$\Lambda: {\mathcal L}_{\mathfrak A}({\mathcal A},{\mathcal
A^*})\longrightarrow {\mathcal L}_{\mathfrak A}({\mathcal
A^{**}},{\mathcal A}^{***})$$ such that for every
 $T\in {\mathcal L}_{\mathfrak A}({\mathcal A},{\mathcal A^*})$, $F,G \in \mathcal A^{**}$,
and bounded nets $(a_j),(b_k)\subset \mathcal A $ with $ \hat a_j
\stackrel{J^{\perp}}{\longrightarrow}F$ and $ \hat b_k
\stackrel{J^{\perp}}{\longrightarrow}G$ \emph{(}where the superscript
$J^\perp$ shows the convergence in the weak topology $\sigma
(\mathcal A^{**},J^\perp)$ generated the family $J^\perp$ of
$w^*$-continuous functionals on $\mathcal A^{**}$\emph{)}, we have
$$\langle\Lambda (T)F,G\rangle= \lim_j\lim_k \langle T(a_j),b_k\rangle.$$
\end{lem}
\begin{proof} We only need to check that $\Lambda $
is a $\mathcal A$-$\mathfrak A$-module homomorphism. It is easy to see
that for $f \in J^{\perp} , \alpha\in {\mathfrak A}$ we have
$f\cdot\alpha,\alpha\cdot f \in J^{\perp}$. If $\alpha \in
\mathfrak A$ then $ \hat b_k\cdot\alpha
\stackrel{J^{\perp}}{\longrightarrow}G\cdot\alpha $, hence
\begin{align*}
\langle\Lambda (\alpha \cdot T)F,G\rangle &=\lim_j\lim_k \langle(\alpha\cdot T)(a_j),b_k\rangle \\
&=\lim_j\lim_k \langle\alpha\cdot T(a_j),b_k\rangle\\
&=\lim_j\lim_k \langle T(a_j),b_k\cdot \alpha\rangle,
\end{align*}
and
\begin{align*}
\langle[\alpha\cdot \Lambda (T)]F,G\rangle &=\langle \Lambda (T)F,G\cdot \alpha\rangle=\lim_j\lim_k \langle T(a_j),b_k\cdot \alpha\rangle.
\end{align*}
Similarly, $\Lambda (T\cdot\alpha)=\Lambda (T)\cdot\alpha$ for all $T\in {\mathcal L}_{\mathfrak A}({\mathcal A},{\mathcal A^*})$ and $\alpha\in \mathfrak A$.
\end{proof}

Consider the map $\lambda : \mathcal A^{**}/N\longrightarrow
\mathcal A^{**}/J^{\perp \perp}; F+N \mapsto F+ J^{\perp \perp}$. It follows from the proof of \cite[Theorem 3.4]{abeh} that $ N \subseteq J^{\perp \perp}$, and thus $\lambda$ is well
defined. Also, it is easy to see that $\lambda$ is a bounded $\mathcal
A$-$\mathfrak A$-module homomorphism.

The following result is the module version of the classical case
which is proved in \cite[Theorem 2.2]{mos}.

\begin{thm}\label{mflat}
{\it If  $\mathcal A^{**}$
is module biflat, then so is $\mathcal A$.}
 \end{thm}

\begin{proof}
 Suppose that $
\widehat{\omega}^* $ has a left inverse $\mathcal A/J$-$\mathfrak
A$-module homomorphism $ \widehat \theta $, then $ \widehat \theta
\circ \widehat{\omega}^* =id_{({\mathcal A^{**}}/N)^*}$. Assume
that $ j:({\mathcal A}/{J})^* \longrightarrow {({\mathcal
A^{**}}/{J^{\perp \perp }})^*}$ is the canonical embedding, and
$\widetilde{\omega}$, $\lambda$ and $\Lambda $ are as the above.
Consider the map $i:{\mathcal A}/{J} \longrightarrow {\mathcal
A^{**}}/N$; ($a+J\mapsto a+N$). Obviously $i$ is well defined.
Let us show that $\Lambda \circ \widetilde{\omega}^* =
\widehat{\omega}^* \circ\lambda^*\circ j $. Take $\varphi \in
({\mathcal A}/{J})^* $, and $F,G\in{\mathcal A}^{**}, \hat a_l
\stackrel{J^{\perp}}{\longrightarrow}F$ and $ \hat b_k
\stackrel{J^{\perp}}{\longrightarrow}G$, then

\begin{align*}
\langle[(\Lambda \circ \widetilde{\omega}^*)(\varphi
)](F),G\rangle &=\lim_l\lim_k
\langle(\widetilde{\omega}^*(\varphi))(a_l),b_k\rangle \\
&=\lim_l\lim_k \varphi(a_lb_k+J)
\end{align*}
and
\begin{align*}
\langle[(\widehat{\omega}^*\circ\lambda^* \circ j)(\varphi
)](F),G\rangle &=\langle \lambda^*(j(\varphi)),F\square G+ N\rangle\\
&=\langle j(\varphi)\circ\lambda,F\square G+ N\rangle\\
&=\langle j(\varphi),F\square G+ J^{\perp \perp}\rangle\\
&=\lim_l\lim_k\langle a_lb_k+J,\varphi\rangle \\
&=\lim_l\lim_k \varphi(a_lb_k+J).
\end{align*}
Put $\Delta=i^* \circ  \widehat \theta \circ \Lambda $. Then it is easy to check that $\Delta$ is a left inverse for
$\widetilde{\omega}^*$.
\end{proof}

\begin{lem} \label{bail}{\it If  $
{\mathcal A^{**}}/{J^{\perp \perp }}= ({\mathcal A}/{J})^{**}$
has a bounded approximate identity, then so does
 ${\mathcal A}/{J}$.}\end{lem}

\begin{proof} The result immediately follows from
\cite[Proposition 28.7]{b} and \cite[Lemma 1.1]{glw}.\end{proof}

The following theorem shows the relation between the concepts of module amenability and of module biflatness on Banach algebras. Since its proof is similar to the proof of \cite[Theorem 2.1]{bod}, is omitted. We only remember that the module amenability of a Banach algebra $\mathcal A$ implies that $\mathcal A/J$ has a bounded approximate identity \cite[Proposition 2.2]{am1}. We employ this result to show that there are semigroup algebras such that their second dual are module biflat but which are not themselves biflat.
\begin{thm} \label{biflat}
 Let $\mathcal A$ be a Banach ${\mathfrak
A}$-bimodule and let $\mathcal A/J$ be a commutative Banach
${\mathfrak A}$-bimodule. Then $\mathcal A $ is module
amenable if and only if $\mathcal A$ is module biflat and $\mathcal A/J$ has a bounded approximate identity.
\end{thm}

In the next result we give a generalization of \cite[Proposition
3.6]{am2} under a weaker condition. In fact if $\mathcal A$ is a
commutative $\mathfrak A$-bimodule, then $\mathcal A/J$ and
$\mathcal A^{**}/N$ are commutative $\mathfrak A$-bimodules.
Therefore $\mathcal A^{**}/J^{\perp \perp}$ is also a commutative
$\mathfrak A$-bimodule (see also \cite[Theorem 2.2]{amb}).

\begin{thm}{\it Let $\mathcal A/J$ and ${\mathcal A}^{**}/N$ be commutative Banach ${\mathfrak A}$-bimodules.
If $\mathcal A^{**}$ is module amenable, then so is $\mathcal A$.}
\end{thm}

\begin{proof} If $\mathcal
A^{**}$ is module amenable, then $\mathcal A^{**}$ is module
biflat and ${\mathcal A}^{**}/N$ has a bounded approximate
identity $\{E_j+N\}$ by Theorem \ref{biflat}. Since $\lambda$ is
surjective, $\{E_j+J^{\perp \perp}\}$ is a bounded approximate
identity for $\mathcal A^{**}/J^{\perp \perp}$. Now the result
follows from Theorem \ref{mflat}, Lemma \ref{bail} and Theorem \ref{biflat}.\end{proof}

\begin{cor}{\it Let $\mathcal A/J$ commutative Banach ${\mathfrak A}$-bimodule and $N$ be weak$^*$-closed in $\mathcal A^{**}$.
If $\mathcal A^{**}$ is module amenable, then so is $\mathcal A$.}
\end{cor}



Let $P$ be a partially ordered set. For $p\in P$, we set
$(p]=\{x:x\leq p\}$. Then $P$ is called
{\it locally finite} if $(p]$ is finite for all $p\in P$, and {\it
locally $C$-finite} for some constant $C>1$ if $|(p]|<C$ for all
$p\in P$. A partially ordered set $P$ which is locally $C$-finite,
for some constant $C$ is called {\it uniformly locally finite}. For example, the semigroup $S=(\mathbb{N}, \bullet)$ in which $m\bullet n=$min$\{m,n\}$ is locally finite but is not uniformly locally finite.

Let $S$ be an inverse semigroup with the set of idempotents $E$. Assume that $J$ is the closed ideal $J$ generated by the set $\{\delta_{set}-\delta_{st} :\, s,t \in S,  e \in E\}.$ By module actions defined in (\ref{e1}), we have
$$\delta_{se}-\delta_{s}\in J \Longleftrightarrow \delta_{set}-\delta_{st}\in J$$
for all $s,t \in S$ and $e \in E$. Therefore $\ell^{1}(S)/J\cong {\ell ^{1}}(S/\approx)$ is always a commutative $ \ell
^{1}(E)$-bimodule.

\begin{ex}
\emph{ Let $S=(\mathbb{N}, \bullet)$ be the semigroup of positive
integers with minimum operation, that is $m\bullet n=$min$\{m,n\}$,
then each element of $\mathbb{N}$ is an idempotent, and thus
$G_{\mathbb{N}}$ is the trivial group with one element.
Therefore ${\ell ^{1}}( \mathbb{N})^{**}$ is ${\ell ^{1}}(
E_{\mathbb{N}})$-bimodule amenable by \cite[Theorem 3.4]{am2}. Let us $\mathcal N$ be the closed ideal of $\ell^{1}(\mathbb{N})^{**}$ generated by $(F\cdot\alpha)\square G-F \square
(\alpha\cdot G)$, for $F,G\in\ell^{1}(\mathbb{N})^{**}$ and $\alpha\in
{{\ell ^{1}}(E_{\mathbb{N}})}$. Since ${\ell ^{1}}( \mathbb{N})/J$ is a commutative Banach ${\ell ^{1}}(E_{\mathbb{N}})$-bimodule, so is ${\ell ^{1}}( \mathbb{N})^{**}/\mathcal N$. Hence  ${\ell ^{1}}( \mathbb{N})^{**}$ is module biflat by Theorem \ref{biflat}.
On the other hand, $(\mathbb{N},\leq)$ is not uniformly locally finite. Thus $\ell^{1}(\mathbb{N})$ is not biflat, and so it is not biprojective
\cite{ram}. Therefore, by \cite[Corollary 3.11]{em}, $\ell^{1}(\mathbb{N})^{**}$ is not biflat.}
\end{ex}

\begin{ex} \emph{Let $\mathcal C$ be the
bicyclic inverse semigroup generated by $p$ and $q$, that is
$$\mathcal C=\{p^mq^n : m,n\geq 0 \},\hspace{0.2cm}(p^mq^n)^*=p^nq^m. $$}
\emph{The set of idempotents of $\mathcal C$ is $E_{\mathcal
C}=\{p^nq^n : n=0,1,...\}$ which is totally ordered with the
following order
$$p^nq^n \leq p^mq^m \Longleftrightarrow m \leq n.$$}
\emph{It is shown in \cite{am2} that ${\ell ^{1}}(\mathcal C)^{**}$ is not module amenable as an ${\ell
^{1}}(E_{\mathcal C})$-bimodule. Therefore it is not module biflat by Theorem \ref{biflat}. Note that ${\ell ^{1}}(\mathcal C)^{**}$ is not biflat by \cite[Corollary 3.11]{em}. However $(E_{\mathcal C},\leq)$ is
not uniformly locally finite, so $\ell ^{1}(\mathcal C)$ is
is not biprojective.}
\end{ex}

\begin{ex} \emph {Let $G$ be a group, and let $I$ be a non-empty set. Then for
$S=\mathcal{M}^0 (G, I)$, the {\it Brandt inverse semigroup}
corresponding to $G$ and the index set $I$, it is shown in
\cite{pou} that $G_S$ is the trivial group. Therefore
${\ell ^{1}}(S)^{**}$ is module amenable, and thus it is module biflat by Theorem \ref{biflat}.  It is obvious that $E_S$ is uniformly locally finite, indeed}
 \begin{equation*}
(s]=\{t\in E_S: t=ts\}= E_SS=
\begin{cases}
\{0,e\}\,\,\hspace{.7cm} \text{if $s\neq 0$}\\

\{0\},\,\,\hspace{0.9cm}\text{if $s=0$.}
\end{cases}
\end{equation*}
 \emph {Hence, $(E_S,\leq)$ is a uniformly locally finite semilattice. Now by \cite[Proposition 2.14]{ram}, it
follows that $(S,\leq)$ is uniformly locally finite. Also each maximal subgroup of $S$ at an idempotent is isomorphic to $G$, hence
${\ell ^{1}}(S)$ is not biflat if $G$ is not amenable \cite[Theorem 3.7]{ram}. In fact, for non-amenable locally compact group $G$, ${\ell ^{1}}(S)^{**}$ is not biflat \cite[Theorem 2.2]{mos}. Note that however, if $I$ is finite, then biflatness of ${\ell ^{1}}(S)^{**}$ is equivalet to the finitness of $G$ \cite[Corollary 3.12]{em}.}
\end{ex}

\section*{ACKNOWLEDGEMENT}

The authors sincerely thank the anonymous reviewers for their careful reading, constructive comments and fruitful suggestions to improve the quality of
the first draft.

\bigskip
\bigskip

{ \textit{\footnotesize Received: Month xx, 200x}}
\begin{flushright}
{\footnotesize\textit{Department of Mathematics} \\ \textit{Garmsar Branch} \\ \textit{Islamic Azad University} \\ \textit{Garmsar,
 Iran} \\ \textit{abasalt.bodaghi@gmail.com} }
\end{flushright}
\begin{flushright}\footnotesize
and
\end{flushright}
\begin{flushright}
\footnotesize {\textit{Young Researchers and Elite Club}}\\ {\textit{Ardabil Branch}}\\
{\textit{Islamic Azad University}} {\textit{Ardabil, Iran}}\\
{\textit{ jabbari\underline{ }al@yahoo.com, ali.jabbari@iauardabil.ac.ir}}\\
\end{flushright}
\end{document}